\documentclass[graybox]{svmult}

\usepackage{type1cm}      

\usepackage{makeidx}        
\usepackage{graphicx}       
                            
\usepackage{multicol}  
\usepackage[bottom]{footmisc}

\usepackage{newtxtext}    
\usepackage[varvw]{newtxmath}

\usepackage{tikz}
\usepackage{wrapfig}
\usepackage{tikz,pgfplots}

\newcommand{\R}{\mathbb{R}}

\makeindex

\begin{document}

\title*{Self-similar axisymmetric flows with swirl}
\author{Theodoros Katsaounis, Ioanna Mousikou and Athanasios E. Tzavaras}
\institute{Theodoros Katsaounis \at University of Crete, Heraklion 71409, Greece and Inst.\ of App.\ and Comp.\ Math.\ (IACM), FORTH, Heraklion 71110, Greece,
\email{thodoros.katsaounis@uoc.gr}
\and Ioanna Mousikou \at King Abdullah University of Science and Technology, Thuwal 23955-6900, Saudi Arabia,
 \email{ioanna.mousikou@kaust.edu.sa}
\and Athanasios E. Tzavaras \at King Abdullah University of Science and Technology, Thuwal 23955-6900, Saudi Arabia,
 \email{athanasios.tzavaras@kaust.edu.sa}}

\maketitle

\abstract{We consider an infinite vortex line in a fluid which interacts with a boundary surface as a simplified model for tornadoes. We study self-similar solutions for stationary axisymmetric Navier-Stokes equations and investigate the types of motion which are compatible with this structure when viscosity is non-negative. For viscosity equal to zero, we construct a class of explicit stationary solutions. We then consider solutions with slip discontinuity and show that they do not exist in this framework.}

\section{Introduction}
\label{sec:1}
Tornadoes are considered among the most extreme and violent weather phenomena on Earth. They can occur under appropriate circumstances in all continents expect Antarctic and can be hazardous causing loss of human lives and extensive properties damages.

Meteorologists define as a tornado a rapidly rotating mass of air that extends downward from a cumuliform cloud, i.e. a cloud formed due to vertical motion of air parcels to the ground. There exist several types of tornadoes, such as landspouts and waterspouts. The majority of the most destructive tornadoes are known as supercell since they are generated within supercell thunderstorms \cite{MR10}, \cite{MR14}.

Due to the complexity of tornadoes, the current knowledge about them comes mainly from laboratory experiments and numerical models of idealized supercell thunderstorms, as Rotunno (2013) stated in \cite{Rot13}. In 1972, Ward \cite{Ward} conducted a pioneering laboratory experiment reproducing a tornado-like flow using a simplified model for a steady flow and a fluid with constant density. Based on this work, several experimental and numerical simulations have taken place and provided important information in the field of fluid dynamics of tornadoes, \cite{Rot13}. Furthermore, various attempts have been made to analytically model a tornado-like flow. Assuming that a vortex line resembles the tornado core, these models are derived using the basic motion of equations of fluid dynamics for an axisymmetric flow, i.e. the axisymmetric Euler and Navier - Stokes equations, for incompressible homogeneous fluids. A detailed presentation can be found in \cite{Kim17} and \cite{Gillmeier-p} and in references therein.

Motivated by the aforementioned vortex models, a different, theoretical approach was introduced by Long (1958, 1961) \cite{Long58}, \cite{Long61}. Considering the existence of an infinite vortex line in a fluid interacting with a plane boundary surface, he presented the reduction of incompressible axisymmetric Navier-Stokes equations to a system of differential equations. Independently, Goldshtik (1960) showed that a similar reduction of incompressible axisymmetric Navier-Stokes equations to a system of differential equations leads to a class of exact self-similar solutions, \cite{Gold60}. Serrin (1972) broadened this class of solutions and described the existence of three different solution profiles depending on an arbitrary parameter and the kinematic viscosity, \cite{Serrin}. There are several studies of mathematical aspects of the aforementioned system of differential equations under other types of boundary conditions, \cite{GS89}, \cite{GS90}, \cite{Gold90}, and also studies of the related subject of conical flows, \cite{SH1999}, \cite{FFA00}, \cite{Shtern12}.

Here, we first develop a class of exact stationary solutions for Euler and Navier-Stokes equations. Afterwards, we consider the problem of whether such solutions can be connected with slip-type discontinuities. If this was the case, it would provide a relation with "two-cell" solutions of Serrin, \cite{Serrin}. We show that they do not exist for the given set of boundary conditions. The same holds true for conical flows. This manuscript is an extract of the work presented in \cite{KMTz} where the connection of such Euler and Navier-Stokes solutions is examined using boundary layer analysis.

\section{Cylindrical Axisymmetric Navier-Stokes Equations}
\subsection{Introduction}
We consider the system of Navier-Stokes equations for an incompressible homogeneous fluid formulated as follows:
\vspace{-7pt}
\begin{subequations}
\label{NS}
\begin{align}
\vec{u}_t + (\vec{u} \cdot \nabla) \vec{u} &= - \nabla p + \nu \, \Delta \vec{u} , \\
\nabla\cdot \vec{u} & = 0,
\end{align}
\end{subequations}
where $\vec{u} : \R^3\times \R_+ \to \R^3$ is the velocity vector of the fluid, $p : \R^3 \times \R_+\to \R$ is pressure and $\nu \ge 0$ is the coefficient of kinematic viscosity. Motivated by the shape of a tornado, we introduce cylindrical coordinates $(r,\theta,z)$
\begin{align*}
x_1 = r \,\cos\theta, \quad x_2 = r\,\sin\theta, \quad x_3 = z,
\end{align*}
and focus on axisymmetric flows, i.e. a flow where the velocity vector $\vec{u} = (u,v,w)$ does not depend on azimuth angle $\theta$. The axisymmetric Navier-Stokes equations take the form
\begin{subequations}
\begin{align}
\label{tNS1}
\frac{\partial u}{\partial t} + u \frac{\partial u}{\partial r} + w \frac{\partial u}{\partial z} - \frac{v^2}{r} & = \nu \Big[\frac{1}{r}\frac{\partial }{\partial r} \Big(r \frac{\partial u}{\partial r}\Big)+ \frac{\partial^2 u}{\partial z^2} -  \frac{u}{r^2}  \Big] - \frac{\partial p}{\partial r} \\
\label{tNS2}
\frac{\partial v}{\partial t} + u \frac{\partial v}{\partial r} + w \frac{\partial v}{\partial z} + \frac{uv}{r} & = \nu \Big[\frac{1}{r}\frac{\partial }{\partial r} \Big(r \frac{\partial v}{\partial r}\Big)+ \frac{\partial^2 v}{\partial z^2} -  \frac{v}{r^2}  \Big] \\
\label{tNS3}
\frac{\partial w}{\partial t} + u \frac{\partial w}{\partial r} + w \frac{\partial w}{\partial z} \qquad  & = \nu \Big[\frac{1}{r}\frac{\partial }{\partial r} \Big(r \frac{\partial w}{\partial r}\Big)+ \frac{\partial^2 w}{\partial z^2}  \qquad \Big] - \frac{\partial p}{\partial z} \\
\label{tNS4}
\frac{1}{r} \frac{\partial }{\partial r} (ru) + \frac{\partial w}{\partial z}  & = 0
\end{align}
\end{subequations}
\subsection{Self-Similar Formulation} 
The Navier-Stokes equations remain invariant under scaling
\begin{equation*}
\vec{u}_\lambda(t,r,z) = \lambda \vec{u}(\lambda^2 t,\lambda r, \lambda z) \quad \textrm{and} \quad p_\lambda(t,r,z) = \lambda^2 p(\lambda^2 t,\lambda r, \lambda z).
\end{equation*} Looking for self-similar solutions and focusing only on stationary flows, we establish the ansatz in $\xi = \frac{z}{r}$
\vspace{-5pt}
\begin{equation*}
\label{ansatz}
u(r,z) = \frac{1}{r} U(\xi), \quad v(r,z) = \frac{1}{r} V(\xi), \quad  w(r,z) =\frac{1}{r}  W(\xi)  \quad \textrm{and} \quad p(r,z) =\frac{1}{r^2}  P(\xi).
\vspace{7pt}
\end{equation*} 
Such an ansatz induces a singularity at $r=0$ which in the applied math literature is considered as the line vortex resembling the tornado core. For convenience, we also introduce a new variable $\theta(\xi)$, namely we set $\theta(\xi) = W -\xi U$, which coincides with the self-similar form of the stream function. After a lengthy calculation, we obtain a system of ordinary differential equations
\vspace{-5pt}
\begin{subequations}
\label{ssform}
\begin{align}
\bigg[\frac{\theta^2}{2} + (1+\xi^2)P \bigg]' &=  \nu \bigg[\xi\theta - (1+\xi^2) \theta' \bigg]' -\xi V^2 \\
V' \theta &= \nu \Big[3\xi  V' + (1+\xi^2) V''\Big] \\
\bigg[\theta^2 - \xi \Big(\frac{\theta^2}{2}\Big)' + P \bigg]' &= \nu \Big[\xi\theta - \xi^2 \theta' - \xi (1+\xi^2) \theta'' \Big]'  \\
\theta' &= - U
\end{align}
\end{subequations}
This is viewed as a coupled system of $\theta(\xi)$, $V(\xi)$ and $P(\xi)$ where $U(\xi) = - \theta'$ and $W(\xi) = \theta - \xi \theta'$. After imposing boundary conditions, the problem can be reformulated as
\vspace{-5pt}
\begin{subequations}
\label{Th-V-eq}
\begin{align}
\frac{\theta^2}{2} - \nu \bigg[(1+\xi^2) \theta' + \xi \theta \bigg]  &=    G(\xi) + \,\, {E_0} \big(\xi\sqrt{1+\xi^2} - \xi^2 \big) \\
\nu V'' + \frac{3\nu\xi - \theta}{1+\xi^2} V' &= 0
\end{align}
\end{subequations}
where
\vspace{-5pt}
\begin{align*}
G(\xi) = \xi\sqrt{1+\xi^2} \int_{\xi}^{\infty} \bigg[\frac{1}{\zeta^2(1+\zeta^2)^\frac{3}{2}} \int_{0}^{\zeta} s V^2(s) \,ds  \bigg] d\zeta.
\end{align*} 
Here we consider no-slip conditions on r-axis, i.e. ${\vec{u} = 0}$ at ${\xi=0}$, and no-penetration condition on z-axis, i.e. ${\vec{u} \cdot \vec{n} = 0}$ as $\mathbf{\xi \to \infty}$. A restriction on swirl $V$ is also added to close the system. Namely, we take $V \to V_\infty$, as $\xi \to \infty$. System \eqref{Th-V-eq} can now be solved numerically. After multiple numerical experiments, we observe that under certain combinations of parameters $\nu, V_\infty, E_0$ there exist three different profiles of solution. In the first case, the flow is directed outward near the plane $z = 0$ and downward near the vortex line. In the second case it is inward near the plane $z = 0$ and upward near the vortex line. For the last case, the flow is directed inward near the plane $z = 0$ and downward near the vortex line. These are in agreement with results presented in \cite{Serrin}. Under a suitable change of variables, i.e. setting $x = \frac{\xi}{\sqrt{1+ \xi^2}}$ and $\bar{\Theta} (x)  = - \sqrt{1+ \xi^2} \, \theta(\xi)$, $\bar{V} (x) = V(\xi)$, system \eqref{Th-V-eq} can be put in a similar form to systems studied by Goldstick-Shtern \cite{GS89} and Serrin \cite{Serrin} starting from a different solution ansatz.

\section{Stationary Euler Equations}
\label{inviscid}
\subsection{Continuous Solution}
Let us consider the case of inviscid Navier-Stokes system, i.e. the case where kinematic viscosity is equal to zero. Therefore, setting $\nu = 0$ into \eqref{ssform}, the system becomes
\begin{subequations}
\begin{align}
\label{theta-eq}
\bigg[\frac{\theta^2}{2} + (1+\xi^2)P \bigg]' &= -\xi V^2 \\
\label{v-eq}
V' \theta &= 0 \\
\label{p-eq}
\bigg[\theta^2 - \xi \Big(\frac{\theta^2}{2}\Big)' + P \bigg]' &= 0
\end{align}
\end{subequations}
Equation $\eqref{v-eq}$ implies that either $\theta(\xi)$ is equal to zero or $V(\xi)$ is a constant function. Supposed that $\theta \neq 0$ and thus $V(\xi)$ is continuous, we have
\begin{equation*}
V \equiv V_{0},
\end{equation*}
where $V_{0}$ is a given constant. This yields to a simple system of differential equation which can be solved analytically. In order to define the constants arising after integration, boundary conditions are imposed. Motivated by the structure of the problem, we consider no-penetration boundary conditions on both axes, i.e. $\vec{u} \cdot \vec{n} = 0$. In other words, we require that the orthogonal component of the velocity vector is equal to zero on the axes, which implies that $W = 0$ at $\xi = 0$ and  $U \to 0$ as  $\xi \to \infty$. Consequently, an explicit family of solutions that depends on parameters $V_0 = V(0)$ and $E_0 = P(0)$ is derived as follows
\begin{align}
\label{sol-eul}
\theta^2 (\xi) &= 2 {k_0} \ \phi(\xi) \quad \textrm{ and }
\quad V (\xi) = V_{0}
\end{align}
where  $\phi(\xi) = \xi \sqrt{1+\xi^2} - \xi^2$ and $k_0 = E_0 + \frac{V_0^2}{2}$ must be a positive constant. Expressions for $U, W$ and $P$ can easily be calculated using the definition of $\theta(\xi)$.

It is worth mentioning that if $\theta$ is positive, then the flow is directed inward near the plane $z = 0$ and upward near the vortex line. Conversely, if $\theta$ is negative, the flow has the reverse direction, i.e it is directed outward near the plane $z = 0$ and downward near the vortex line, see Fig.1. Such behaviors also occur when solving Navier - Stokes equations, \cite{Serrin}. 

\begin{figure}[htbp]
\label{fig:th}
   \begin{minipage}{0.5\textwidth}
     \centering
     \includegraphics[width=.9\linewidth]{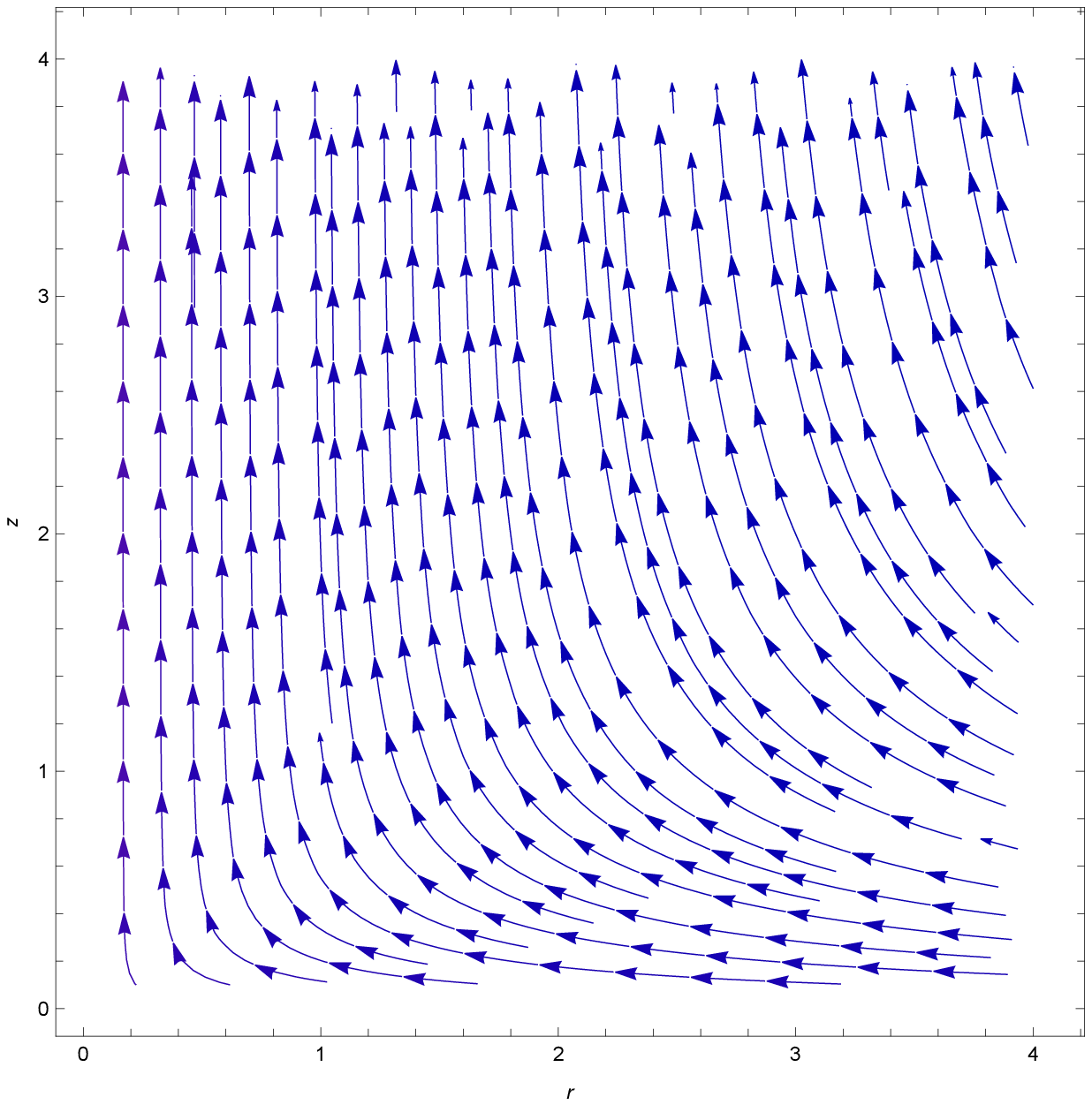}
   \end{minipage}\hfill
   \begin{minipage}{0.5\textwidth}
     \centering
     \includegraphics[width=.9\linewidth]{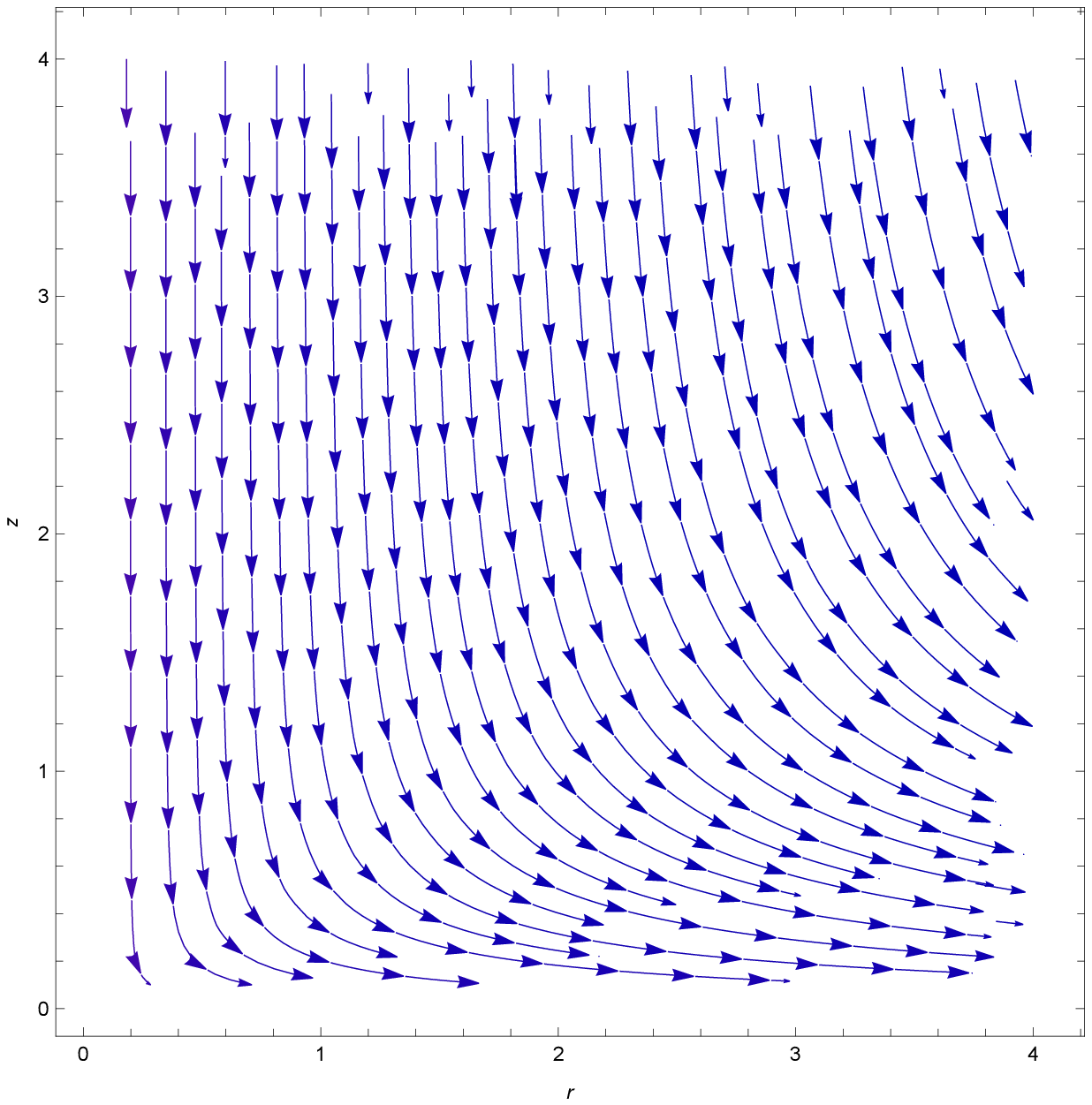}
   \end{minipage}
    \caption{Velocity vector field $(u,w)$ in $(r,z)$ plane for $V_0 = 1$ and $E_0 = 1$. Left, $\theta>0$; right, $\theta<0$}
    \vspace{-15pt}
\end{figure}

\subsection{Discontinuous Solutions}
Although the flow patterns described in the previous section coincide with flows derived using the stationary Navier-Stokes equations, the interesting case where $\theta$ changes sign and thus flow changes direction is not observed. To investigate whether this phenomenon is feasible, we assume that a solution of $\eqref{theta-eq} - \eqref{p-eq}$ has a discontinuity at some point $\xi = \sigma$, for $\sigma \in (0,\infty)$. Hence, we introduce an ansatz
\begin{align}
\label{disc}
\theta (\xi) =
\left\{
\begin{aligned}
& \theta_{-}\,\, , \quad &&\xi \in (0,\sigma) \\
&\theta_{+} \,\, , \quad &&\xi \in (\sigma,\infty)
\end{aligned}
\right. 
\quad \text{and} \quad
V (\xi) =
\left\{
\begin{aligned}
& V_- \,\, , \quad &&\xi \in (0,\sigma) \\
& V_+ \,\, , \quad &&\xi \in (\sigma,\infty)
\end{aligned}
\right.
\end{align}
and examine if such solutions satisfy the Rankine-Hugoniot jump conditions associated with the system $\eqref{theta-eq} - \eqref{p-eq}$.

\subsubsection{Jump Conditions}
Let $(\theta,V,P)$ be a (generally weak) self-similar solution of Euler equations which satisfies the system of ordinary differential equations \eqref{theta-eq} - \eqref{p-eq} in the sense of distributions. Under a suitable choice of test function, the weak form of the system can be expressed over an interval $(a,b) \subset (0,\infty)$ as follows  
\vspace{-7pt}
\begin{subequations}
\label{weak2}
\begin{align}
\bigg(\frac{\theta^2(\xi)}{2} + (1+\xi^2)P(\xi)  \bigg)\Bigg|_{a-}^{b_+}  &= - \int_a^b \zeta V^2(\zeta) d\zeta, \\
\bigg(\theta(\xi) V(\xi)\bigg) \Bigg|_{a-}^{b_+} &= - \int_a^b U(\xi) V(\xi) d\xi, \ \\
\bigg(\theta^2 - \xi \Big(\frac{\theta^2}{2}\Big)' + P(\xi) \bigg)\Bigg|_{a-}^{b_+} &= 0, \\
\theta(\xi) \Big|_{a-}^{b_+} &= - \int_a^b U(\xi)  d\xi,
\end{align}
\end{subequations}
\begin{proposition}
Let $(\theta, V, P)$ of class $\theta\in W^{1,1}((0,\infty))$, $V \in BV((0,\infty)) \cap L^{\infty}((0,\infty))$ and $P \in BV((0,\infty)) \cap L^{\infty}((0,\infty))$. Then, equations \eqref{weak2} are  satisfied on every $(a,b) \subset (0,\infty)$.
\end{proposition}
\begin{proof}
See \cite{KMTz}
\end{proof} 

Consider now a solution $(\theta, V, P)$ of class described in Proposition 1, which has a discontinuity at some point $\xi = \sigma$ and is defined in form \eqref{disc}. Due to its regularity, the right and left limits of the solution exist along the discontinuity. Equations \eqref{weak2} then provide the following jump conditions at $\xi = \sigma$
\begin{subequations}
\begin{align*}
\frac{1}{2}\Big(\theta_+^2 - \theta_-^2\Big) + (1+\xi^2)\Big(P_+ - P_-\Big)  &= 0 \\
\theta_+ V_+ - \theta_- V_- &= 0 \ \\
\bigg(\theta_+^2 - \xi \Big(\frac{\theta_+^2}{2}\Big)' - \Big(\theta_-^2 - \xi \Big(\frac{\theta_-^2}{2}\Big)'\Big) \bigg)  + \bigg(P_+ - P_- \bigg)  &= 0 \\
\theta_+ - \theta_-  &= 0
\end{align*}
\end{subequations}
where $(\theta\pm,V\pm, P\pm)$ denotes the one-sided limits. The last equation implies that $\theta$ must be continuous for any $\xi \in (0, \infty)$. Hence, jump conditions reduce to
\vspace{-7pt}
\begin{subequations}
\label{jump}
\begin{align}
\Big[P \Big] = 0 \\
\Big[\theta V \Big] = 0 \\
\Big[\theta^2 -  \xi \theta \, \theta' \Big] = 0 
\end{align}
\end{subequations}
For $\theta(\xi) \neq 0$ $\forall \xi \in (0,\infty)$, \eqref{jump} yields that $V$ and $P$ are also continuous. However, this is not the case if ${\theta}(\sigma)=0$. From \eqref{jump}, it implies $P(\xi)$ is continuous for any $\xi$ while $V(\xi)$ and $\theta'(\xi)$ have a jump discontinuity at $\xi = \sigma$. 

\subsubsection{Existence of Discontinuous Solutions}
Let us consider a solution of $ \eqref{theta-eq} - \eqref{p-eq}$ in form \eqref{disc}. Under the continuity restrictions of $\theta$, i.e. ${\theta}_+(\sigma) = {\theta}_-(\sigma) = 0$ and  no-penetration boundary conditions on the axes, i.e. $\theta-(0) = 0,\, \theta+(\xi) \to 0  \textrm{ as }\xi \to \infty$, the discontinuous solution becomes
\vspace{-5pt}
\begin{subequations}
\label{disc-sol}
\begin{align}
\frac{\theta^2(\xi)}{2}  &&&=
\left\{
\begin{aligned}
& {k_-} \Bigg[\phi(\xi) -\phi(\sigma) -  \frac{ \phi(\sigma)}{\sigma^2} (\xi^2 - \sigma^2)\Bigg] \,\, , &&\xi \in (0,\sigma) \\
& {k_+} \bigg[ \phi(\xi) - \phi(\sigma) \bigg] \,\, , &&\xi \in (\sigma,\infty)
\end{aligned}
\right.
\\
V (\xi) &&&=
\left\{
\begin{aligned}
& V_- \,\, , \hspace*{135pt} &&\xi \in (0,\sigma) \\
& V_+ \,\, , \quad &&\xi \in (\sigma,\infty)
\end{aligned}
\right.
\end{align}
\end{subequations}
where $k_+, k_-$  as well as $V_+, V_-$ are constants. The discontinuity restrictions for $\theta'(\xi)$ leads to the following theorem.

\begin{theorem}
Let $(\theta, V, P)$ be a weak solution of \eqref{theta-eq} - \eqref{p-eq} of class $\theta\in W^{1,1}((0,\infty))$, $V \in BV((0,\infty)) \cap L^{\infty}((0,\infty))$ and $P \in BV((0,\infty)) \cap L^{\infty}((0,\infty))$ which satisfies the boundary condition 
\begin{equation*}
V(0+)=V_0, \quad P(0+) = E_0, \quad \theta(0) = 0, \quad \theta'(\infty)=0.
\end{equation*}
There does not exists a solution $(\theta, V, P)$ with a discontinuity at a single point that fulfills the jump conditions \eqref{jump}. 
\end{theorem}

\begin{proof}
Suppose $\theta$ is given in form \eqref{disc-sol}. From jump condition \eqref{jump}, we have 
\begin{equation}
\label{j2}
\big({k_-}-k_+\big) \phi'(\sigma) = 2{k_-}\frac{ \phi(\sigma)}{\sigma}\quad \Rightarrow \quad \frac{k_+}{k_-}= 1 - 2\frac{\phi(\sigma)}{\sigma \, \phi'(\sigma)}
\end{equation}
which provides an additional relation for constants $k_+,k_-$, with the right hand-side to be negative. We want to check if this relation is compatible with sign restrictions for constants $k_+,k_-$. By construction $k_+$ is always positive since $\phi(\xi)$ is a non-negative function. Therefore, it is sufficient to examine the sign of $k_-$ by finding the sign of $\theta_-^2$. For $\xi\in(0,\sigma)$, set
\vspace{-5pt}
\begin{equation*}
J(\xi) = \phi(\xi) -  \phi(\sigma) -  \frac{\phi(\sigma)}{\sigma^2} (\xi^2 - \sigma^2)  
\end{equation*}
We observe that $J(0)= J(\sigma) = 0$, $J'(0) = \phi'(0)>0$ and $J''<0$. This implies that $J(\xi)>0 \forall \xi\in(0,\sigma)$, and thus $k_-$ is also positive. So, we get a contradiction.
\end{proof}

\section{Conical Flows}
\begin{wrapfigure}{r}{0.35\textwidth}
\vspace{-40pt}
\begin{tikzpicture}[xscale=3, yscale=10]
\draw [<-] (0,0.2) -- (0,0) ; 
\draw [dashed,->]  (0,0) -- (1.1,0.0); 
\node [left] at (0,0.2) {$z$};
\node [below] at (1.1,0) {$r$};
\draw (0,0) -- (1,0.18); 
\node [below] at (1,0.15) {$\xi = \xi_0>0$}; 
\draw [blue] (0,0) -- (0.7,0.2); 
\node [below] at (0.8,0.21) {$\xi = \sigma$}; 
\draw (0,0) -- (1,-0.12); 
\node [below] at (1,-0.125) {$\xi = \xi_0<0$}; 
\end{tikzpicture}
\caption{Conical shaped domain}
\vspace{-60pt}
\end{wrapfigure}

Motivated by the study of Euler equations presented in the previous section, we are interested in extending it for the case of axisymmetric conical flows, i.e. for flows in a cone-shaped domain. Suppose there exists $\xi_0 \in \R$, we seek solutions of $\eqref{theta-eq} - \eqref{p-eq}$ defined over the interval $[\xi_0, \infty)$.

\subsection{Continuous Solutions}
Let us begin with the case where solutions are continuous. As before, we assume $\theta \neq 0$ and $V_0 = V(\xi_0)$. If no-penetration boundary conditions are imposed on both ends of the domain $[\xi_0, \infty)$, we get the conditions:
\vspace{-5pt}
\begin{align*}
W(\xi_0) - \xi_0 U(\xi_0) = 0& \text{ at } \xi = \xi_0, \\
U(\xi) \to 0& \text{ as  }  \xi \to \infty
\end{align*}
Therefore, solutions of $\eqref{theta-eq} - \eqref{p-eq}$ for all $\xi $ in a conical domain $[\xi_0, \infty)$ become
\vspace{-5pt}
\begin{align}
\label{sol-con}
\frac{\theta^2}{2}(\xi) &=  k_0 \bigg(\phi(\xi) - \phi(\xi_0)\bigg)  \quad \textrm{ and }
\quad V (\xi) = V_{0}
\end{align}
where $E_0 = \frac{1}{\sqrt{1+\xi_0^2}+\xi_0} \Big( P(\xi_0) \sqrt{1+\xi_0^2}-\frac{V_0^2}{2}\xi_0\Big)$ and $k_0=\frac{1}{1-2\phi(\xi_0)}\Big(\frac{V_0^2}{2} + E_0\Big)>0$.
\subsection{Discontinuous Solutions}
To investigate now the existence of discontinuous solutions, we consider a solution of $ \eqref{theta-eq} - \eqref{p-eq}$ with a discontinuity at some point $\xi = \sigma$, for $\sigma \in (\xi_0,\infty)$. Under the restriction of continuity of $\theta$ at $\xi=\sigma$, i.e. ${\theta}_+(\sigma) = {\theta}_-(\sigma) = 0$, and no-penetration boundary conditions, the discontinuous solution takes the form
\vspace{-5pt}
\begin{subequations}
\label{th-con}
\begin{align}
\frac{\theta^2(\xi)}{2} &&&=
\left\{
\begin{aligned}
& {k_-} \Bigg[\bigg(\phi(\xi) -\phi(\sigma)\bigg) -  \frac{\phi(\xi_0) -  \phi(\sigma)}{\xi_0^2-\sigma^2} \,\, (\xi^2 - \sigma^2)\Bigg] \,\, , &&\xi \in (\xi_0,\sigma) \\
& {k_+} \bigg[ \phi(\xi) - \phi(\sigma) \bigg] \,\, , &&\xi \in (\sigma,\infty)
\end{aligned}
\right.
\\
V (\xi) &&&=
\left\{
\begin{aligned}
& \quad V_- \,\, , \hspace*{170pt} &&\xi \in (\xi_0,\sigma) \\
& \quad V_+ \,\, , \quad &&\xi \in (\sigma,\infty)
\end{aligned}
\right.
\end{align} 
\end{subequations} 
where $k_+$, $k_-, V_-, V_+$ are constants. The discontinuity restrictions for $\theta'(\xi)$ leads to the following theorem.
\begin{theorem}
Let $(\theta, V, P)$ be a weak solution of $\eqref{theta-eq} - \eqref{p-eq}$ defined on a conical domain of class $\theta\in W^{1,1}((\xi_0,\infty))$, $V \in BV((\xi_0,\infty)) \cap L^{\infty}((\xi_0,\infty))$ and $P \in BV((\xi_0,\infty)) \cap L^{\infty}((\xi_0,\infty))$ which satisfies the boundary conditions 
\begin{equation*}
V(\xi_0+)=V_0, \quad P(\xi_0+)=E_0+\Big(\frac{V_0^2}{2}+E_0\Big) \frac{\xi_0}{\sqrt{1+\xi_0^2}}, \quad \theta(\xi_0) = 0, \quad \theta'(\infty)=0.
\end{equation*}
There does not exists a solution $(\theta, V, P)$ with a discontinuity at a single point that fulfills the jump conditions \eqref{jump}.
\end{theorem}

\begin{proof}
Suppose there exists $\theta$ expressed as \eqref{th-con}. Because of jump conditions \eqref{jump}, we request
\begin{equation}
\label{jump-con}
\frac{k_+}{k_-}= 1 - 2 \, \frac{\phi(\xi_0) -  \phi(\sigma)}{\xi_0^2-\sigma^2} \,\, \frac{\sigma}{\phi'(\sigma)} \quad \Rightarrow \quad \frac{\xi_0+\sigma}{2\sigma} \bigg(\frac{k_+}{k_-} \bigg)= \frac{\xi_0+\sigma}{2\sigma} - \frac{1}{\phi'(\sigma)} \frac{\phi(\xi_0) -  \phi(\sigma)}{\xi_0-\sigma} 
\end{equation}
As before, it is sufficient to check if this relation is compatible with sign restrictions for constants $k_+,k_-$. Since $\phi$ is decreasing, it is clear that $k_+$ is positive for all $\xi \in (\sigma,\infty)$.
\begin{itemize}
\item  {\underline{Case 1:} $\xi_0>0$} \\
If $\xi_0 < \sigma$, we get that the right hand-side of \eqref{jump-con} is negative. Therefore, it is satisfied if $k_-$ is also negative. To find this, we check the sign of $\theta_-^2$. Set
\begin{align*}
J_{con}(\xi) &= \phi(\xi) -\phi(\sigma) -  \frac{\phi(\xi_0) -  \phi(\sigma)}{\xi_0^2-\sigma^2} \,\, (\xi^2 - \sigma^2) = (\xi^2 - \sigma^2) \bigg(F(\xi) - F(\xi_0)\bigg)
\end{align*}
where $F(\xi) = \frac{\phi(\xi) -  \phi(\sigma)}{\xi^2-\sigma^2}$. Using that $F(\xi)$ is a decreasing function and $\xi_0 < \sigma$, we get that $J(\xi)$ is positive. This implies that $k_-$ is positive and leads to contradiction.
\item {\underline{Case 2:} $\xi_0<0$} \\
We consider first the instance where $|\xi_0|<\sigma$. This is equivalent to case $1$ described above. So, let us move to the instance where $|\xi_0|>\sigma$. From \eqref{jump-con}, we have that the right hand-side of the above relation is negative. Since $\frac{\xi_0+\sigma}{2\sigma}<0$, \eqref{jump-con} is satisfied if $k_-$ is positive. To find this, we check again the sign of $\theta_-^2$. It is clear that $(\xi^2 - \sigma^2)>0$ for $\xi \in (-|\xi_0|,\sigma)$. Since $F(\xi)$ is a decreasing function, we conclude that $J_{con}(\xi)$ is negative and as consequence $k_-$ is also negative. This also leads to contradiction.
\end{itemize}
\end{proof}

%
%
%

\end{document}